\documentclass[10t,a4paper]{amsart}
\usepackage[utf8]{inputenc}
\usepackage[T1]{fontenc}

\usepackage[foot]{amsaddr}

\title[Transitive graphs quasi-isometric to planar graphs]{A Note on quasi-transitive graphs \\ quasi-isometric to planar (Cayley) graphs}
\author{Joseph Paul MacManus}
\date{July 18, 2024}

\address{Mathematical Institute,  University of Oxford, Oxford, OX2 6GG, UK}
\email{macmanus@maths.ox.ac.uk}

\usepackage[numbers]{natbib}

\usepackage{amsfonts}
\usepackage{amsmath}
\usepackage{amsthm}
\usepackage{amssymb}
\usepackage{stmaryrd}
\usepackage{mathrsfs}  
\usepackage{thmtools}

\usepackage{tgpagella}
\usepackage[colorlinks]{hyperref} 
\hypersetup{
    linkcolor=black,
    citecolor=blue,
}
\usepackage{fancyref}
\usepackage{todonotes}

\DeclareMathOperator{\length}{length}
\DeclareMathOperator{\Isom}{Isom}

\DeclareMathOperator{\Homeo}{Homeo}

\newcommand{\R}{\mathbf{R}}
\newcommand{\Z}{\mathbf{Z}}
\newcommand{\N}{\mathbf{N}}

\newcommand{\C}{\mathbf{C}}

\newcommand{\into}{\hookrightarrow}

\newcommand{\pres}[2]{\langle #1 \ ; \ #2 \rangle}

\newtheorem{theorem}{Theorem}[section]

\newtheorem*{theorem*}{Theorem}

\newtheorem{proposition}[theorem]{Proposition}

\theoremstyle{definition}

\newtheorem{remark}[theorem]{Remark}

\makeatletter
\newcommand{\myitem}[1]{%
\item[#1]\protected@edef\@currentlabel{#1}%
}
\makeatother

\begin{document}

\begin{abstract}
    Given a connected, locally finite, quasi-transitive graph $X$ which is quasi-isometric to a planar graph $\Gamma$, we remark that one can upgrade $\Gamma$ to be a planar Cayley graph, answering a question raised by Esperet--Giocanti and Hamann.
\end{abstract}

\maketitle




\section*{Introduction}

The topic of `coarse graph theory' is a fast-growing area, where the goal is to understand how `large-scale' assumptions affect the geometry of graphs. 
This point-of-view in graph theory was recently popularised by Georgakopoulos and Papasoglu, and we direct the interested reader to their article for further reading \cite{georgakopoulos2023graph}.  

One particular point of interest has been in understanding those graphs which are quasi-isometric to planar graphs. 
For example, it was shown by the author in \cite{macmanus2023accessibility} that if a quasi-transitive graph is quasi-isometric to a planar graph then it is necessarily accessible, in the sense of Thomassen--Woess \cite{thomassen1993vertex}. By combining this statement with a deep theorem originating in the work of Mess \cite{mess1988seifert}, it is further shown that if a Cayley graph is quasi-isometric to a planar graph, then it is quasi-isometric to a planar Cayley graph. 

Recently, it was shown by Esperet and Giocanti that quasi-transitive graphs which are `minor-excluded' are quasi-isometric to planar graphs \cite{esperet2023coarse}. There, they raise the question as to whether one can take the aforementioned planar graph to be itself quasi-transitive, or possibly to be a Cayley graph.
In \cite{hamann2024quasi} Hamann shows that one can upgrade the planar graph in their theorem to indeed be quasi-transitive, but raises the more general question of if given \textit{any} quasi-transitive graph which is quasi-isometric to a planar graph $\Gamma$, can we similarly upgrade $\Gamma$?
In this short note, we settle the above and record the following fact. 

\begin{restatable}{alphtheorem}{main}\label{thm:main}
Let $X$ be a connected, locally finite, quasi-transitive graph. Suppose $X$ is quasi-isometric to a planar graph. Then $X$ is quasi-isometric to a planar \textbf{Cayley} graph.
\end{restatable}

More precisely, such a graph is quasi-isometric to either $\{1\}$, $\Z$, $\Z^2$, a cocompact Fuchsian group, or some free product of the aforementioned groups.
It is also possible to further upgrade the planar Cayley graph in Theorem~\ref{thm:main} to be that of a \textit{Kleinian function group}. See Remark~\ref{rmk:function-groups} below for more details. 

It was a long-standing question of Woess \cite[Prob.~1]{woess1991topological} as to whether every vertex-transitive graph is quasi-isometric to a Cayley graph. A counterexample was eventually presented in \cite{eskin2012coarse}. Somehow, the geometry of our setting means that a similar counterexample is impossible here.

\medskip

The proof of Theorem~\ref{thm:main} proceeds case-wise. First considering one-ended graphs, we separately consider those which are hyperbolic and those which are not, in the sense of Gromov. In the hyperbolic case, we combine the convergence theorem of Tukia, Gabai, and Casson--Jungreis with a corollary of the Montgomery--Zippin theorem to deduce the desired conclusion. In the non-hyperbolic case, we appeal to an inequality of Varopoulos to deduce that our graph has a quadratic growth function. Following this, a theorem of Woess \cite{woess1994topological} implies that our graph is quasi-isometric to $\Z^2$. The infinite-ended case is then deduced by combining a theorem of Hamann \cite{hamann2018tree} characterising quasi-isometries between accessible graphs, together with a theorem of the author stating that the graphs in question are indeed accessible \cite{macmanus2023accessibility}.

\subsection*{Conventions}

Every graph in this note is assumed to be simplicial. We take $0$ to be a natural number.

\section{One-ended graphs}

Throughout this section, let $X$ be a connected, locally finite, quasi-transitive, one-ended graph which is quasi-isometric to some planar graph. 
We will assume a working knowledge of the theory of isometry groups of graphs, viewed as topological groups. We point the reader to \cite{cornulier2014metric} as a good modern reference. The survey article of Woess \cite{woess1991topological} also serves as a good introductory text. 

\subsection{Hyperbolic graphs} In this subsection we assume that $X$ is hyperbolic, in the sense of Gromov. 
To avoid confusion we will call a metric space \textit{Gromov-}hyperbolic, while a finitely generated group shall be said to be \textit{word-}hyperbolic. That is, a finitely generated group is word-hyperbolic if it admits a Gromov-hyperbolic Cayley graph. 

By \cite[Thm.~B]{macmanus2023accessibility}, we have that $X$ is quasi-isometric to a complete Riemannian plane. It is a standard fact that the boundary of a Gromov-hyperbolic complete Riemannian plane is a circle. 
Let $G = \Isom(X)$ denote the isometry group of $X$, which acts faithfully and cocompactly on $X$. We have that $G$ is a totally disconnected, locally compact topological group. Consider the kernel $N \leq G$ of the action of $G$ on $\partial_\infty X$.  This is a compact normal subgroup of $G$, and $H := G /N$ acts faithfully and continuously on the circle. Every quotient of a locally compact group is locally compact, so $H$ (or an index-two subgroup of $H$) is a locally compact subgroup of $\Homeo_+(S^1)$. By a corollary of the Montgomery-Zippin theorem \cite[Thm.~4.7]{ghys2001groups}, it follows that $H$ is a Lie group. Since $H$ is totally disconnected, we deduce that it is discrete.

Let $Y = X/N$ be the quotient graph. Note that since $N$ is a compact subgroup of $\Isom(X)$ we have that $N$ acts on $X$ with bounded orbits, and so the quotient map $\pi : X \to  Y$ is a quasi-isometry. It is easy to check that $H$ acts faithfully and continuously on $Y$. It follows that the quotient group $H = G/N$ homeomorphically embeds as a topological group into $\Isom(Y)$ \cite[Prop.~5.B.6]{cornulier2014metric}. Since $H$ is discrete, we deduce that the action of $H$ on $Y$ is properly discontinuous \cite[Prop.~5.B.10]{cornulier2014metric}. Since $Y/H$ is compact,  so we see by the Švarc--Milnor lemma that $H$ is quasi-isometric to $Y$ (and thus also to $X$). Therefore $H$ is a word-hyperbolic group with a circular boundary. By the convergence group theorem of Tukia, Gabai, and Casson--Jungreis \cite{tukia1988homeomorphic, gabai1992convergence, casson1994convergence} we deduce that $H$ is virtually Fuchsian. We have thus shown the following. 

\begin{theorem}\label{thm:hyp}
	Let $X$ be a connected, locally finite, quasi-transitive graph. Suppose $X$ is one-ended, Gromov-hyperbolic, and quasi-isometric to a planar graph. Then $X$ is quasi-isometric to a Fuchsian group, and thus to a planar Cayley graph. 
\end{theorem}

\subsection{Non-hyperbolic graphs} Now, we assume that $X$ is not hyperbolic. In this case it will be helpful to assume without loss of generality that $X$ is vertex-transitive. This is harmless, as it is a standard fact that every connected, locally finite, quasi-transitive graph is quasi-isometric to a connected, vertex-transitive, locally finite graph \cite[Prop.~5.1]{woess1994topological}.

Proceeding in the spirit of an argument presented in \cite[\S~3]{papasoglu2005quasi}, we will make use of an isoperimetric inequality of Varopoulos, as it appears in \cite[Thm.~2.1]{saloff1995isoperimetric}.

Given our transitive graph $X$, let $V(n)$ denote the size of the closed $n$-ball around a given vertex. In other words, $V : \N \to \N$ is the \textit{growth function} of $X$. Given a subset $\Omega \subset VX$, we denote by $\partial \Omega$ the set of vertices which lie at a distance of exactly 1 from $\Omega$. 

\begin{proposition}
	The vertex-transitive graph $X$ satisfies the following isoperimetric inequality.
	For all finite  $\Omega \subset VX$, we have that 
	$$
	\frac{|\Omega|}{\phi(2|\Omega|)} \leq 4|\partial\Omega|,
	$$
	where $\phi(\lambda) = \inf\{n \in \N : V(n) > \lambda\}$. 
\end{proposition}

The statement above has the easy consequence that if there exists $C, D \in \N$ such that $V(n) \geq Cn^D$ for all $n \geq 1$, then there exists $C'> 0$ such that 
$$
|\Omega|^{\tfrac{D}{D-1}} \leq C' |\partial \Omega|.
$$
for all finite  $\Omega \subset VX$. 
The contrapositive of this says that if we can find `large' subsets with a `small' boundary, then this subsequently bounds the growth of $X$.

By \cite[Thm.~B]{macmanus2023accessibility}, we know that $X$ is quasi-isometric to a complete Riemannian plane $P$ which is not Gromov-hyperbolic. Let $K$ be some triangulation of $P$, such that the inclusion $K \into P$ is a quasi-isometry.\footnote{Note that the existence of such a triangulation is non-obvious and requires proof. However, the plane $P$ constructed in the proof of Theorem~B in \cite{macmanus2023accessibility} comes naturally equipped with such a triangulation, so we shall ignore this technicality.} 

In what follows, a \textit{local $(\lambda,\varepsilon)$-quasi-geodesic} is a closed loop $\ell$ such that any subpath of length less than $\length(\ell)/2$ is a $(\lambda, \varepsilon)$-quasi-geodesic.
A theorem of Papasoglu states that a connected, locally finite graph is Gromov-hyperbolic if and only if it has uniformly thin bigons \cite{papasoglu1995strongly}. As noted in \cite[Lem.~3.8]{papasoglu2005quasi}, this easily implies that $K$ contains arbitrarily long simple loops which are uniformly locally $(\lambda, \varepsilon)$-quasi-geodesic. 
Let $\ell \subset K$ be such a loop. Then by the Jordan curve theorem, $\ell$ separates $P$ into a bounded and an unbounded component. Let $U \subset P$ denote the bounded component, so $U$ is an open disk. It is an immediate consequence of the fact that $\ell$ is a local $(\lambda, \varepsilon)$-quasi-geodesic that $U$ must contain a vertex $v$ of $K$ which lies at a distance of at least $c\cdot \length(\ell)$ from $\ell$, where $c > 0$ is some constant depending only on $\lambda$, $\varepsilon$, and $K$. 
Let $\psi : K \to X$ be a quasi-isometry. Let $y = \psi(v)$. It is easy to check that there must exist some uniform $m > 0$ such that $y$ is contained in a bounded connected component of $X - N_m(\psi(\ell))$ (provided $\ell$ is taken to be sufficiently long). 
In other words, if we write $n = \length(\ell)$ and let $\Omega \subset VX$ denote the vertex set of $U'$, we see that 
$$
|\partial \Omega| \leq c_1 n, \ \ \ V(c_2n) \leq |\Omega|.
$$
where $c_1, c_2 > 0$ are some uniform constants.
By repeating this construction with arbitrarily long $\ell$, we deduce by Varopoulos' inequality that $X$ has at most quadratic growth.

Building on deep results of Gromov \cite{gromov1981groups} and Bass \cite{bass1972degree}, a theorem of Woess \cite[Thm.~4.1]{woess1994topological} implies our desired conclusion. 

\begin{theorem}[Woess]
	Let $X$ be a connected, locally finite, quasi-transitive graph with quadratic growth. Then $G = \Isom(X)$ contains a compact normal subgroup $N$ such that $G/N$ is a finitely generated discrete group with a finite index subgroup isomorphic to $\Z^2$ acting properly discontinuously on the quotient graph $X/N$.
\end{theorem}

As noted in the hyperbolic case, the quotient map $X \to X/N$ is a quasi-isometry, so  by the Švarc--Milnor lemma we have that $X$ is quasi-isometric to $\Z^2$. We thus deduce the following. 

\begin{theorem}\label{thm:nhyp}
	Let $X$ be a one-ended, connected, locally finite, quasi-transitive graph which is not Gromov-hyperbolic. If $X$ is quasi-isometric to a complete Riemannian plane, then $X$ is quasi-isometric to $\Z^2$, and thus to a planar Cayley graph. 
\end{theorem}

\section{Infinite-ended graphs}

We will need the following theorem of Hamann \cite{hamann2018tree}, which is inspired by the analogous result for finitely generated groups due to Papasoglu and Whyte \cite{papasoglu2002quasi}.

\begin{theorem}[Hamann]\label{thm:hamann}
	Let $X$, $Y$ be connected, locally finite, infinite-ended, quasi-transitive graphs. Let $(T, \mathcal V)$ and $(S, \mathcal U)$ be canonical tree decompositions of $X$, $Y$ with bounded adhesion and tight separators, such that each torso has at most one end.
	Let $\Upsilon_X$, $\Upsilon_Y$ denote the set of quasi-isometry classes of the one-ended torsos of the aforementioned tree decompositions, without repetitions. 
	Then $X$ and $Y$ are quasi-isometric if and only if $\Upsilon_X = \Upsilon_Y$. 
\end{theorem}

We direct the reader to \cite{hamann2018tree} for definitions and statements relating to (canonical) tree decompositions. 
Our main result now follows easily from combining all that has come before. 

\main*

\begin{proof}
	If $X$ is finite or two-ended then there is nothing to prove. If $X$ is one-ended then we are done by Theorems~\ref{thm:hyp} and \ref{thm:nhyp}. 
	
	Suppose then that $X$ is infinite-ended. By the main theorem of \cite{macmanus2023accessibility} we have that $X$ is accessible. More precisely, by \cite[Cor.~C]{macmanus2023accessibility} we have that $X$ admits a canonical tree decomposition with bounded adhesion and tight separators, where each infinite torso is quasi-isometric to a complete Riemannian plane. If $\Upsilon$ denotes the set of quasi-isometry classes of these one-ended torsos, then by Theorems~\ref{thm:hyp} and \ref{thm:nhyp} we have that $\Upsilon$ contains at most two elements, picked from the QI-classes of the hyperbolic plane $\mathbf H^2$ and the Euclidean plane $\R^2$. Combining this observation with Theorem~\ref{thm:hamann}, we deduce that $X$ is quasi-isometric to one of
	$$
	F_2, \ \  \Z^2 \ast \Z^2, \ \ \Sigma \ast \Sigma, \ \  \Z^2 \ast \Sigma, 
	$$
	where $\Sigma$ denotes some fixed cocompact Fuchsian group. The $F_2$-case arises exactly when every torso is finite. In particular, we deduce that $X$ is quasi-isometric to a planar Cayley graph. 
\end{proof}

\begin{remark}\label{rmk:function-groups}
	We conclude by remarking that every planar Cayley graph is quasi-isometric to a `Kleinian function group'. 
	
	Recall that a \textit{Kleinian} group $G$ is a discrete subgroup of $\rm{PSL}(2, \C)$. Such a group admits a canonical action on $\mathbf S^2$ by M\"obius transformations. The unique maximal open subspace $\Omega$ of $\mathbf S^2$ on which $G$ acts properly discontinuously is called the \textit{domain of discontinuity}. If $\Omega$ contains a $G$-invariant connected component then we call $G$ a \textit{function group}. These groups can be completely classified, see \cite{marden1974geometry} or \cite{maskit2012kleinian} for further reading. 
	
	There is a natural combinatorial characterisation of function groups in terms of their Cayley graphs.  
	Given a Cayley graph $\Gamma$ of a finitely generated group $G$, a planar embedding $\vartheta : \Gamma \into \mathbf S^2$  is said to be \textit{covariant} if the action of $G$ on $\Gamma$ sends facial paths to facial paths. 
	By a theorem of Levinson--Maskit \cite{levinson1975special}, the class of finitely generated groups group admitting a covariant embedding into $\mathbf S^2$ corresponds precisely to the Kleinian function groups (see also \cite{georgakopoulos2020planar}).  
	
	Whitney's planar embedding theorem \cite{whitney1992congruent, richter20023} states that 3-connected\footnote{Recall that a graph $\Gamma$ is said to be \textit{3-connected} if it is connected and the removal of any vertex or pair of vertices cannot disconnect $\Gamma$.} planar graphs embed uniquely into 2-sphere, up to post-composition with homeomorphisms of the 2-sphere. Therefore, a 3-connected planar Cayley graph is necessarily covariantly embedded. Thus, it is sufficient to find examples of 3-connected planar Cayley graphs in each quasi-isometry class. The only non-trivial case to consider is infinite-ended groups which are not virtually free. 
	
	These can be constructed by amalgamating triangle groups along certain cyclic subgroups. 
	For example, consider the Euclidean (2,3,6)-triangle group 
	$$
	H = \pres{r, s, t}{(rs)^2, \ (st)^3, \ (rt)^6}.
	$$
	The Cayley graph corresponding to the above presentation is planar. The order-three cyclic subgroup $C = \langle s t \rangle \leq H$ lies on the boundary of a single face. This implies that the Cayley graph of the double $G = H \ast_C H$ given by the presentation
	$$
	G = \pres{a, b, c, r, s, t}{(r s)^2, \ (st)^3, \ (rt)^6,\ (ab)^2, \ (bc)^3, \ (ac)^6, \ bc = st}
	$$
	is planar. Since $|C| \geq 3$ we have that this Cayley graph is also 3-connected. Thus $G$ is a Kleinian function group by the above discussion. 
	
	The theorem of Papasoglu and Whyte \cite{papasoglu2002quasi} implies that $G$  is indeed quasi-isometric to $\Z^2 \ast \Z^2$.  
	A similar construction is possible within the other infinite-ended quasi-isometry classes, which we leave as an exercise to the interested reader.
\end{remark}

\subsection*{Acknowledgements}

I am very grateful to Yves Cornulier for taking the time to share some insight into the structure of isometry groups of hyperbolic graphs with circular boundaries, and also to Louis Esperet, Agelos Georgakopoulos, Matthias Hamann, and Panos Papasoglu for comments on an earlier version of this note. I also thank Lawk Mineh for providing some very sobering counterexamples.

\bibliographystyle{abbrv}
\bibliography{references}

\end{document}